\documentclass[reqno,10pt]{article}
\usepackage{graphicx}
\usepackage{amsmath}
\usepackage{amssymb}
\usepackage{amsthm}
\usepackage{enumitem}
\usepackage{bbm}
\usepackage[l2tabu,orthodox]{nag}
\usepackage[all,error]{onlyamsmath}
\usepackage[strict]{csquotes}
\usepackage{url}
\usepackage{xcolor}

\newtheorem{theorem}{Theorem}
\newtheorem{lemma}[theorem]{Lemma}

\theoremstyle{definition}

\numberwithin{equation}{section}
\numberwithin{theorem}{section}

\newenvironment{OMabstract}{\noindent\textbf{Abstract.} }{\medskip}
\newenvironment{OMsubjclass}{\noindent\textbf{Mathematics Subject Classification (2020):} }{\medskip}
\newenvironment{OMkeywords}{\noindent\textbf{Keywords:}  }{\medskip}

\begin{document}

\author{Dorota Bors, Robert Stańczy} 
\title{Generalized Kolmogorov systems with applications to astrophysics and biology}
\maketitle


\begin{OMabstract}
    We consider generalized Kolmogorov system. We prove the existence of heteroclinic trajectory. We apply the results to astrophysical models for self-gravitating particles and predator-prey systems.
\end{OMabstract}

\begin{OMkeywords}
    Kolmogorov system, Lyapunov function, dynamical system.
\end{OMkeywords}

\begin{OMsubjclass}
    35Q85, 70K05, 85A05.
\end{OMsubjclass}


\section{Introduction}  
Consider the generalized Kolmogorov system
\begin{equation}\label{Kol}
	\begin{cases}
		x'=g(x)G(x,y),\\
		y'=h(y)H(x,y)\,.
	\end{cases}
\end{equation}
where $G(w,z)=H(w,z)=0$ for some $w>0,z>0$. The case where $g(x)=x$ and $h(y)=y$ was originally considered by Kolmogorov, and a discussion of its development can be found in \cite{bibref7}. In this work, we prove that under appropriate monotonicity assumptions on the functions $G$ and $H$, the equilibrium point $(w,z)$ globally attracts all  orbits in the first quadrant, established via the construction of a Lyapunov function. Moreover, under additional assumptions, a heteroclinic connection from the saddle point $(0,0)$ emerges. Specifically, the unstable manifold of the origin forms a heteroclinic orbit that tends toward the global attractor $(w,z)$, as do all other trajectories in the positive quadrant. We also obtain a bound on the heteroclinic trajectory with respect to the first variable. This bound follows from a suitable trapping region, defined as the Lyapunov sublevel set, for a specific triangular set of initial data. Finally, these results are applied in an astrophysical context to the mass--radius limit both for Smoluchowski (cf. \cite{CMP, JDE}) and Einstein (\cite{MMAS, DCDS}) equations. The system is also analyzed from the perspective of predator-prey models.

\subsection{Lyapunov function} 

\begin{theorem} \label{MyThmLabel} 
Assume that $x,y \ge 0$. Suppose that $G_x\cdot H_x^z \le 0$ and $H_y\cdot G_y^w\ge 0$, where we denote  $H^z=H(x,z)$ and $G^w=G(w,y)$, while with subscripts partial derivatives are denoted. Then, the function
\[L(x,y)={\mathcal H}(x)+{\mathcal G}(y)\] 
is a Lyapunov function for (\ref{Kol}), where ${\mathcal H}(w)={\mathcal G}(z)=0$ and ${\mathcal G, H}>0$ elsewhere. The derivatives of these functions are defined as ${\mathcal H}'(x)=H(x,z)/g(x)$, ${\mathcal G}'(y)=-G(w,y)/h(y)$. 
\end{theorem}
\begin{proof}
Multiplying the first equation of the system (\ref{Kol}) by $H(x,z)/g(x)$ and the second one by $-G(w,y)/h(y)$, and summing the results, yields
\[
x'H(x,z)/g(x)-y'G(w,y)/h(y)=G(x,y)(H^z-H(w,z))-H(x,y)(G^w-G(w,z))\,.
\]
By adding and subtracting the term $G(w,y)H(x,z)$ on the right hand side, and using the Mean Value Theorem alongside the definitions of ${\mathcal G}$ and ${\mathcal H}$, the derivative of the Lyapunov function along the orbit of (\ref{Kol}) $(L(x,y))'=x'L_x+y'L_y=x'{\mathcal G}'(x)+{\mathcal H}'(y)$, can be rephrased as
\[
(L(x,y))'=G_xH_x^z(x-w)^2-H_yG_y^w(y-z)^2\,,
\]
Given our initial assumptions regarding the signs of these partial derivatives, it follows that $(L(x,y))'\le 0$, which completes the proof.
\end{proof}

\begin{theorem} \label{MyThm} 
Assume that $x,y \ge 0$. Suppose that $G_x\cdot H_x^z \ge 0$ and $H_y\cdot G_y^w\le 0$, where we denote  $H^z=H(x,z)$ and $G^w=G(w,y)$. Then, the function
\[L(x,y)={\mathcal H}(x)+{\mathcal G}(y)\] 
 is a Lyapunov function 
where  ${\mathcal H}(w)={\mathcal G}(z)=0$ and ${\mathcal G, H}>0$ elsewhere. The derivatives of these functions are defined as ${\mathcal H}'(x)=-H(x,z)/g(x)$ and ${\mathcal G}'(y)=G(w,y)/h(y)$. 
\end{theorem}
\begin{proof}
Multiply the first equation of the system (\ref{Kol}) by $-H(x,z)/g(x)$ and the second by $G(w,y)/h(y)$, and summing the results, yields
\[
-x'H^z/g(x)+y'G^w/h(y)=-G(x,y)(H^z-H(w,z))+H(x,y)(G^w-G(w,z))\,.
\]
By adding and subtracting the term $G(w,y)H(x,z)$ on the right hand side, and utilizing the Mean Value Theorem alongside the definitions of ${\mathcal G}$ and ${\mathcal H}$, the derivative of the Lypaunov function along the orbits of the system (\ref{Kol}) $(L(x,y))'=x'{\mathcal H}'(x)+y'{\mathcal G}'(y)$, can be rephrased as 
\[
(L(x,y))'=-G_xH_x^z(x-w)^2+H_yG_y^w(y-z)^2\,.
\]
Given our initial assumptions regarding the signs of these partial derivatives ($G_x\cdot H_x^z \ge 0$ and $H_y\cdot G_y^w\le 0$), it follows that $(L(x,y))'\le 0$ which completes the proof.
\end{proof}

\subsection{Linearization}
In this section, we analyze the linearization of the generalized Kolmogorov system (\ref{Kol}) at its stationary points.
\begin{lemma} Assume that $G(w,z)=G(0,0)=0$ and $H(w,z)=0=h(0)$. Furthermore, let  $G,H, g$ and $h$ be strictly positive elsewhere. The Jacobian matrix of the system (\ref{Kol}), evaluated at the equilibria $(w,z)$ and $(0,0)$, takes the reduced form
\[
\begin{bmatrix}
gG_x&gG_y\\
hH_x&hH_y+h'H\\
\end{bmatrix}
\]
Provided that at $(w,z)$ we have $G_x<0$, $H_x<0$, $G_y>0$, and $H_y\le 0$, this matrix is:
\begin{itemize}
\item negative definite at $(w,z)$,
\item indefinite at $(0,0)$. 
\end{itemize}
\end{lemma}
\begin{proof}
The Jacobian matrix of the system (\ref{Kol}) evaluated at an arbitrary point is given by
\[
\begin{bmatrix}
g'G+gG_x&gG_y\\
hH_x&h'H+hH_y\\
\end{bmatrix}
\]
Notice that the term $g'G$ vanishes both at $(0,0)$ and at $(w,z)$ because $G(0,0)=G(w,z)=0$. At the equilibrium $(w,z)$, we additionally have $H(w,z)=0$, causing the $h'H$ term to vanish. The matrix simplifies, and we can observe that the top-left entry is $gG_x<0$. The determinant of the matrix at $(w,z)$ is equal to $gh(G_xH_y-H_xG_y)>0$.
Conversely, at the origin $(0,0)$, we have $h(0)=0$. This eliminates the bottom-left entry and the $hH_y$ term, meaning the determinant of the matrix simplifies to $gh'G_xH$. Because $G_x<0$ and the remaining functions are evaluated as positive, this determinant is strictly negative, which confirms the matrix is indefinite at $(0,0)$.
\end{proof}

\begin{lemma}\label{Lem}
Assume the assumptions from the previous Lemma. Let $c=\lim_{t\to -\infty}{y(t)/x(t)}$ for any trajectory starting from the interior of the first quadrant. Then
\[
c=\frac{H(0,0)h'(0)(g(0))^{-1}-G_x(0,0)}{G_y(0,0)}\,.
\]
\begin{proof}
This is immediate application of the l'Hospital rule for the ratio $y/x$. Indeed, using the system (\ref{Kol}), we obtain
\[
c=\lim_{t\to -\infty}\frac{y(t)}{x(t)}=\lim_{t\to -\infty}\frac{y'(t)}{x'(t)}=\lim_{t\to -\infty}\frac{hH}{gG}=\lim_{t\to -\infty}\frac{(h'H+hH_y)y'+hH_xx'}{(g'G+gG_x)x'+gG_yy'}\,.
\]
Then dividing both numerator and denominator by $x'$ we get the relation equivalent to the claim ($c\neq 0$ since we are not on the $y=0$ axis)
\[
g(0)(G_x(0,0)+cG_y(0,0))=h'(0)H(0,0)\,.
\]
\end{proof}
\end{lemma}

Now we are ready to formulate the main theorem on the bound of the heteroclinic trajectory. 

\begin{theorem}
Let the assumptions of the previous Lemma hold.
Assume that for some $w\in (0,\delta)$ with $\delta>0$ and positive $z$, the following hold:
\begin{itemize}
\item[] {\bf Monotonicity}: $G_x, H_x, G_y$ are strictly negative while $H_y\ge 0$.
\item[] {\bf Stationary points}: $h(0)=0$, $H(w,z)=0$, and $G(0,0)=G(w,z)=0$, with $g, h$ positive elsewhere.
\item[] {\bf Unstable tangent}: $hH \le cgG$ along the ray $y=cx$ where $c>0$ is defined as in Lemma \ref{Lem}.
\item[] {\bf Derivative sign}: ${\mathcal H}'(x)=-H(x,z)/g(x)>0$ for $cv>x>w$, where $c>0$ is defined as in Lemma \ref{Lem}.
\item[] {\bf Isoclines}: The equation $G(x,y)$ implicitly defines a unique strictly increasing function $x_+:[0,z]\mapsto [0,w]$. The equation $H(x,y)=0$ defines a unique nonincreasing function $x_-:[z,cv]\mapsto [v,w]$, where $c>0$ is given by Lemma \ref{Lem}.
\end{itemize}
Then, the heteroclinic trajectory is bounded in the first variable by $X$, where
\[
X={\mathcal H}|_{[w,\delta)}^{-1}({\mathcal G}(cv))\,,
\]
given ${\mathcal H}(w)={\mathcal G}(z)=0$, with ${\mathcal G}$ and ${\mathcal H}$ strictly positive elsewhere,  while ${\mathcal H}'=-H^z/g$ and ${\mathcal G}'=G^w/h$.
\end{theorem}
\begin{figure}[h]
\includegraphics[height=6cm]{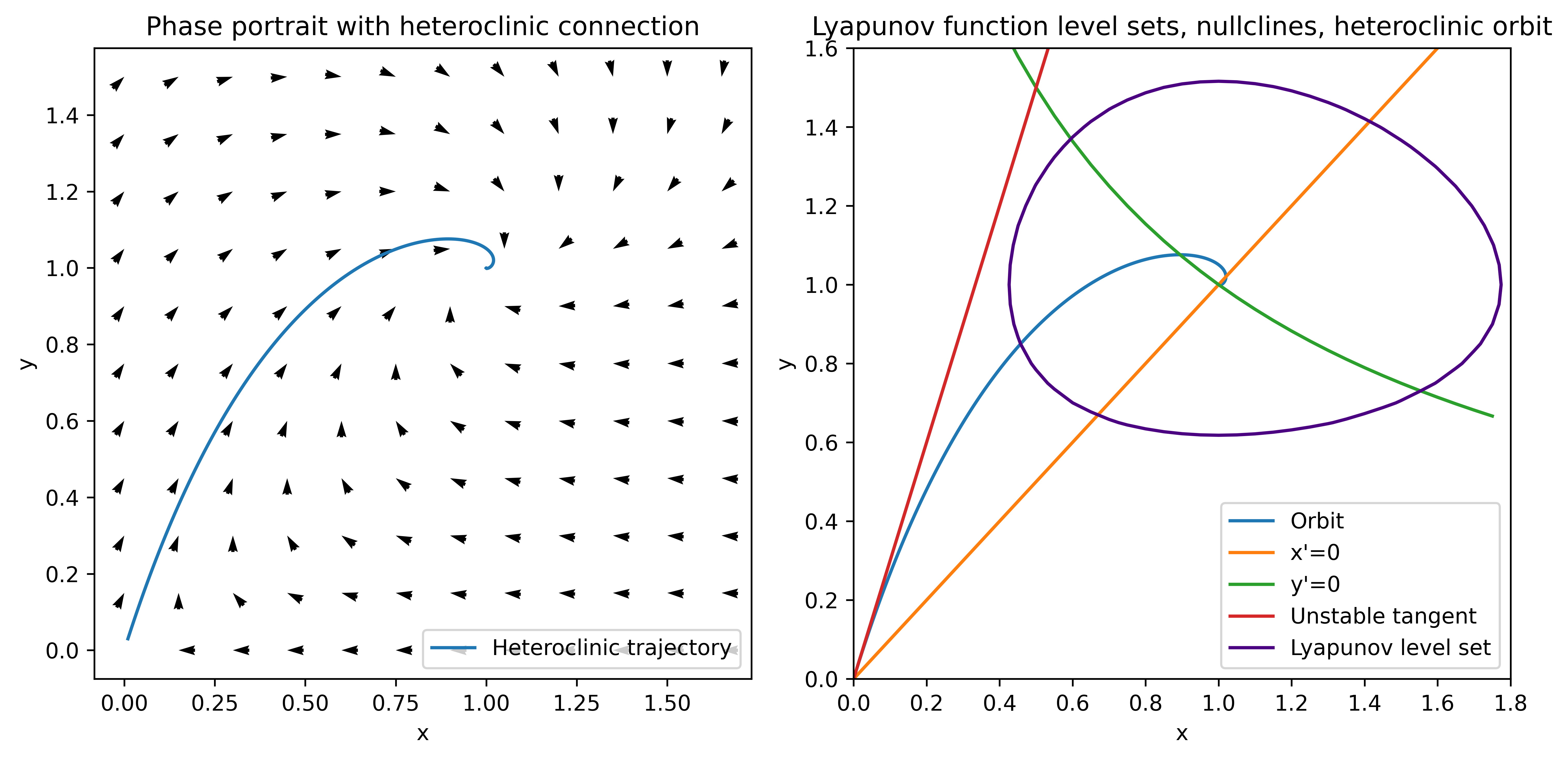}
\caption{Typical phase portrait demonstrating the intersection of the isoclines at the equilibrium $(w,z)$. The $y'=0$ isocline intersects the unstable tangent of the heteroclinic orbit at the point $(v,cv)$.}
\end{figure}
\begin{proof}
Let the point $v$ be defined such that $H(v,cv)=0$. We define $X$ as the maximal x-coordinate solution to the Lyapunov function level equation
\[
L(X,z)=L(w,cv)\,.
\]
Because $L(x,y)={\mathcal H}(X)+{\mathcal G}(z)$, this implies
\[
{\mathcal H}(X)+{\mathcal G}(z)={\mathcal H}(w)+{\mathcal G}(cv)
\]
The explicit formula for $X$ follows directly from this equality, the zero properties of the functions at the equilibrium, and the strict monotonicity of ${\mathcal H}$ on the interval $[w,\delta)$.
\end{proof}

\section{Application to astrophysics}
In this section we shall show that the general theory presented in previous section can be applied to astrophysical model, both in relativistic and classical cases, cf. \cite{BHN, CMP, JDE, MMAS, DCDS, bibref1, S5K, TMNA}. This implies a bound for the mass-radius ratio in the original system describing particles, involving both Einstein equation (resulting in Tolman-Oppenheimer-Volkoff equation) and the Smoluchowski-Poisson equations.
Consider $H(x,y)=a(x)-b(x)y$ and $h(y)=y$ while $G(x,y)=y-x$ and $g(x)=1$ like in \cite{TMNA} where special case was treated. Then the Lyapunov function, for $A'=a$, $B'=b$, $a(z)=zb(z)$, and $A(z)=B(z)=0$, reads
\[
L=zB(x)-A(x)+y-z-z\log(y/z)\,.
\]
which is consistent with our calculations as ${\mathcal H}(x)=zB(x)-A(x)$ and ${\mathcal G}(y)=y-z-\log(y/z)$. Moreover, the estimate obtained in \cite{TMNA},
\[
X={\mathcal H}|_{[w,\delta}^{-1}((a(0)+1)w-z-z\log ((a(0)+1)w/z))
\]
is consistent with the estimate from our main theorem. 
\subsection{Classical case}
The estimate presented here can be understood as the Newtonian limit for the integrated density of the particles expressed by the Smoluchowski-Poisson equation in radial symmetry and a static case, cf. \cite{BHN, CMP, JDE}.
Let $a(x)=2-x$ and $b(x)=0$. Then $A(x)=2x-2-x^2/2$ related by $A'=a$ and the Lyapunov function reads
\[
2L=(x-2)^2+2y-4\log(y)+4\log 2-4\,.
\]
The function $2{\mathcal H}(x)=(x-2)^2$ and ${\mathcal G}(y)=y-2-2\log (y/2)$. Therefore 
\[
X=2+2\sqrt{2-\log 3}<4\,.
\]

\subsection{Relativistic case}
Here, spherically symmetric and static solutions to the Einstein equations lead to the Tolman-Oppenheimer-Volkoff equation. After using Milne variables, we obtain the system with the data considered below.
Moreover, set
$(1-8\pi x)a(x)=2-24\pi x$ and $(1-8 \pi x)b(x)=8\pi$ as for the relativistic case treated in \cite{MMAS, DCDS, TMNA}. 
Then
${\mathcal G}=y-(1+\ln(16\pi y))/(16\pi )$ and
\[16\pi {\mathcal H}=-48\pi x-3\log (1-8\pi x)+3-3\log(2).\]
Finally, 
$16\pi {\mathcal H}(x) = -3+3\log (-W_{-1}^{-1}(16\pi x -2))$, where $W_{-1}$ is the lower branch of the Lambert W  function (or product log), defined as the inverse function of $W_{-1}^{-1}(y)=y\exp(y)$. Consequently, 
\[16\pi {\mathcal H} ^{-1}(y)=2+W_{-1}(-\exp(1+16\pi y/3)).\]
Therefore, the bound for the heteroclinic trajectory reads $2X=2+W(-2^{1/3}e^{-4/3})$, as in \cite{TMNA}.
\begin{figure}[h]
\includegraphics[height=5.2cm]{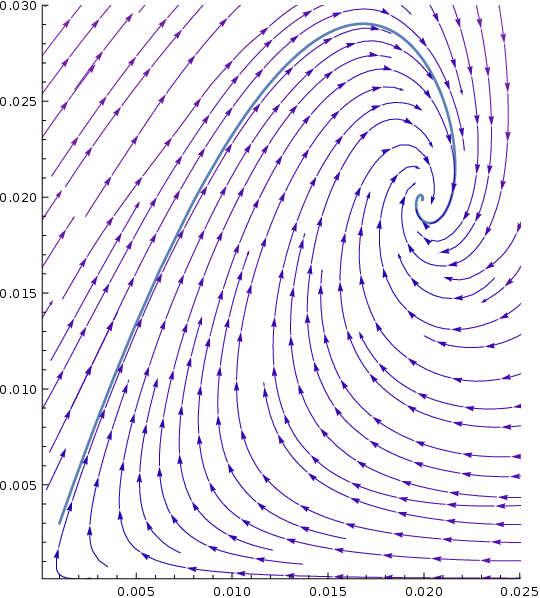}
\includegraphics[height=5.2cm]{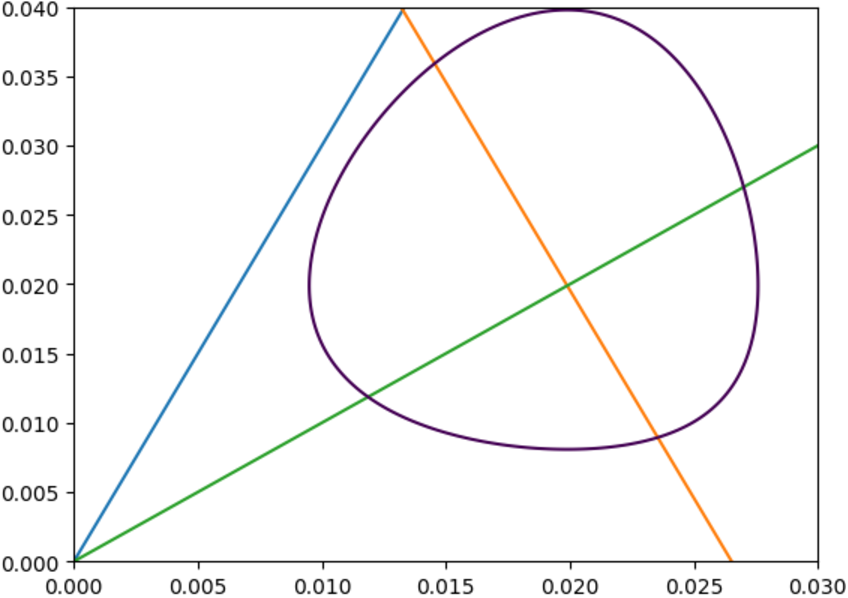}
\caption{Heteroclinic trajectory with phase portrait. In the right panel: \\\textcolor{blue}{unstable tangent manifold}, \textcolor{orange}{isocline y'=0}, \textcolor{olive}{isocline x'=0}, \textcolor{violet}{the Lyapunov level set}.}
\end{figure}

\section{Application to biology}
As announced in the introduction, this section arises from the generalized Kolmogorov system considered in special case in \cite{bibref7}, while the motivation arise from \cite{MI}, utilizing theory borrowed from \cite{bibref2}.
\subsection{Predator-prey model I}
Consider $H(x,y)=\alpha /(1+\kappa x)-\beta yx$ and $h(y)=y$ while $G(x,y)=y-x$ and $g(x)=1$. Then for $\alpha=6$ and $\kappa=\beta=2$ we have $c=7$ and $3=7v^2(1+2v)$. Moreover, 
\[
{\mathcal H}(x)=x^2-1+3\log 3- 3\log(2x+1)
\]
and
\[
{\mathcal G}(y)=y-\log y -1
\]
whence the following Lyapunov function formula follows
\[L=x^2-1-3\log(2x+1)+3\log(3)+y-1-\log(y)\]

\begin{figure}[h]
\includegraphics[height=6cm]{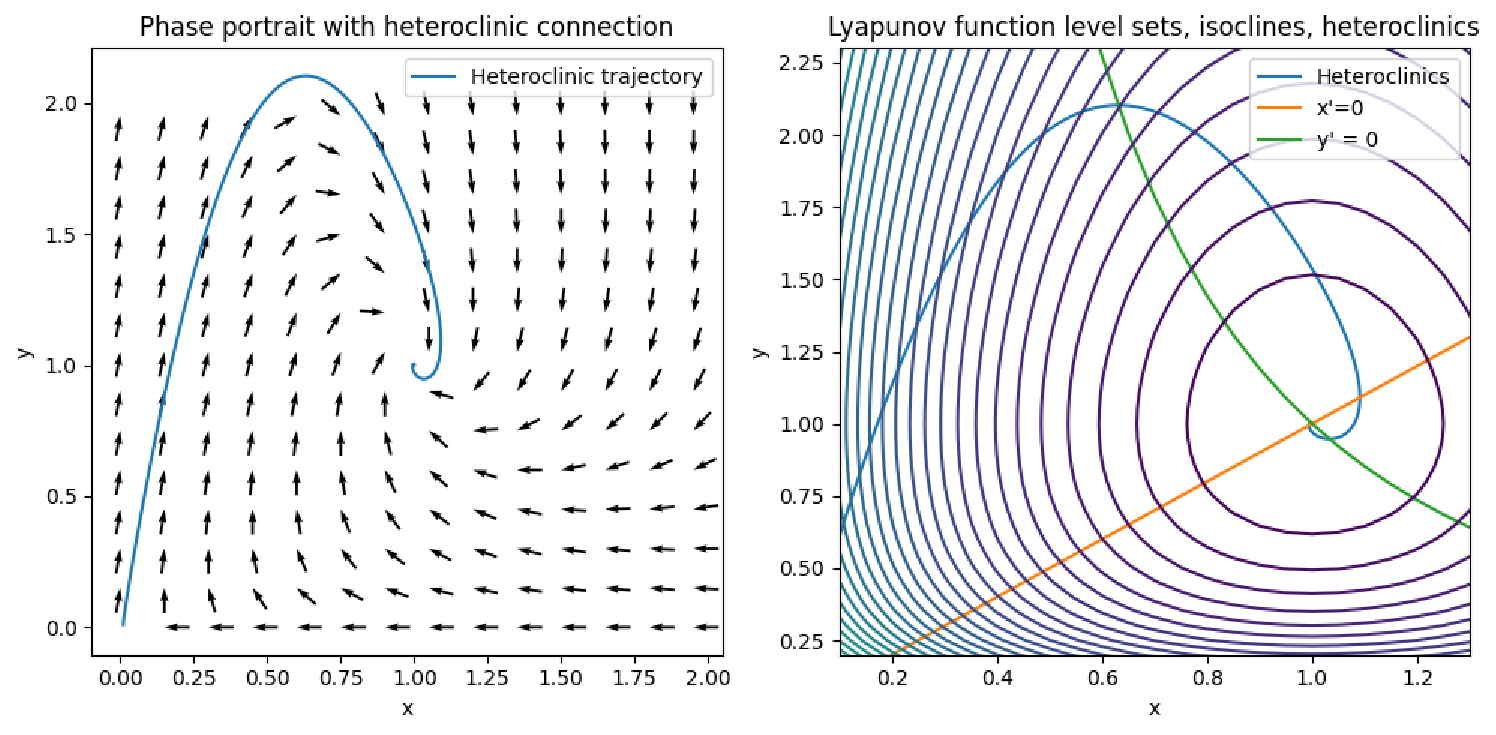}
\caption{Figures for $\beta=\kappa=2$ and $\alpha=6$ in predator-prey model}
\end{figure}

\subsection{Predator-prey model II}
\begin{figure}[h]
\includegraphics[height=6cm]{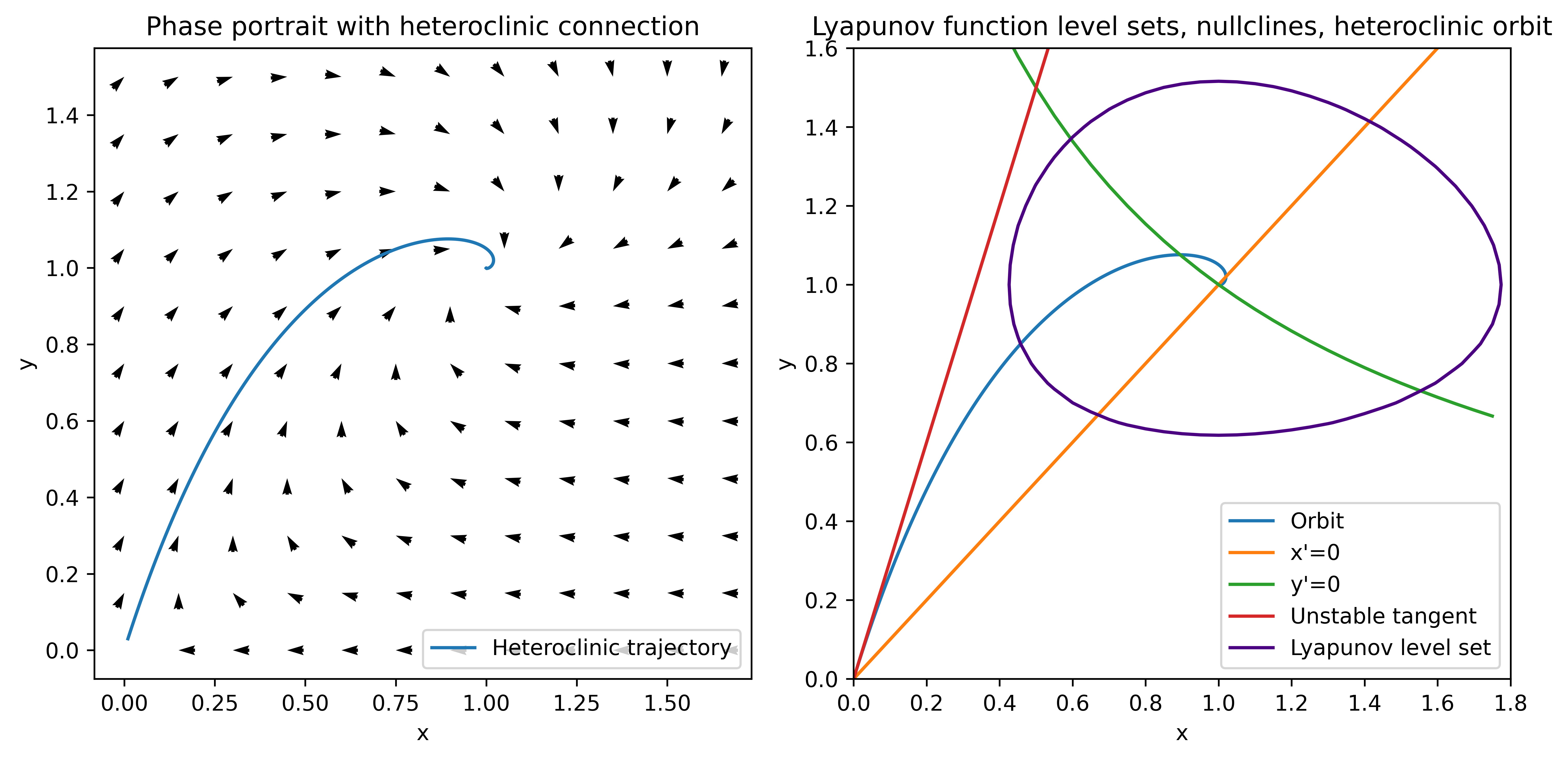}
\caption{Predator-prey model with bound by Lyapunov function}
\end{figure}

Let $H=2/(1+2x)-2y/3$ and $G=y-x$. Moreover, $g=1$ and $h=y$.
Note that then $c=3$ is the tangent of the angle at which heteroclinic trajectory starts from $(0,0)$ and ends at $(1,1)$. Next, we calculate $H(v,3v)=0$ to get $2v=1$. Consequently using ${\mathcal G}(y)=y-\log(y)-1$ and $3{\mathcal H}(x)=2x-2+3\log3-3\log(2x+1)$ with ${\mathcal H}(X)={\mathcal G}(3v)$, we obtain
\[
2X=-3W_{-1}\left( -\exp(-1-{\mathcal G}(3/2))\right)-1
\]
where $W_{-1}$ is the lower branch of the Lambert W function. Substituting ${\mathcal G}(3/2)=1/2-\log(3/2)$, we obtain
\[
X=-\frac{3}{2}W_{-1}\left( -\left(\frac32\right) e^{-3/2}\right)-\frac12=1.75.
\]
This can be interpreted as the bound for the heteroclinic trajectory and for all data starting in the neighbourhood of zero under this separatrix. This can be understood as a control on the number of predators if, at the beginning, their number is less than the number of prey but not less than $1/3$ the prey's number.

To understand the dynamics of $x$, let's look at what happens if the prey goes extinct ($y=0$): $x'=-x$.
The predator $x$ cannot survive without the prey.
\begin{itemize}
\item {\bf Natural Mortality}: The $-x$ term represents the predator's natural death rate. Without prey, the predator population decays exponentially to zero.
\item {\bf Prey-Dependent Reproduction}: The $+y$ term is the predator's growth rate. This is slightly unusual compared to the classic Lotka-Volterra model (which uses an $xy$ "encounter rate" term and will be addressed in the next section). Because growth is dictated simply by $+y$, it implies the predator's reproduction is directly proportional to the sheer abundance of prey in the environment, rather than the random rate at which individual predators and prey bump into each other.
\end{itemize}

\noindent
Recall the Prey Equation: $y'=6y/(1+2x)-2y^2$. To understand dynamics of $y$, let's first look at what happens if the predator goes extinct ($x=0$): $y'=6y-2y^2$. The prey population $y$ can survive and grow on its own. Specifically:
\begin{itemize}
\item {\bf Logistic Growth}: In the absence of predators, the prey undergoes standard logistic growth. We can factor it as 
$y'=2y(3-y)$, which means it has an intrinsic growth rate of $6$ and an environmental carrying capacity of $3$.
\item {\bf Predator Suppression}: Now look at the interaction term  $6/(1+2x)$. As the predator population $x$ increases, the denominator grows, causing this entire term to shrink. The presence of the predator suppresses the prey's intrinsic growth rate. Because the predator is in the denominator, this models a "saturation" effect—the predators interfere with each other when their numbers are high, meaning their negative impact on the prey has a mathematical limit.
\end{itemize}

\subsection{Predator-prey model III}
In this section we shall consider more realistic predator-prey model in which the growth of predator is dependent on their number. The general approach for Kolmogorov system shall apply, yielding Lyapunov function and one point global attractor, yet no heteroclinic trajectory can be found in this case. Consider  $g(x) = \gamma x, h(y) = \delta y, G(x,y)=\delta y-\gamma$ and $H(x,y) = \alpha (1-y/m) - \alpha x\,.$ Then stationary solution apart from $(0,0)$ reads $(w,z)=(1-\gamma/(m\delta),\gamma/\delta)$.

We start by rewriting the formula as follows:

    \begin{align}
        \left\{
        \begin{array}{l}
        x' = -\gamma x + \delta xy, \\
        y' = \alpha (1-y/m)y - \alpha xy.
        \end{array}
        \right.
    \end{align}

    Then ${\mathcal G}(y)=y-\frac{\gamma}{\delta}\log (\frac{\delta y}{\gamma})-\gamma/\delta$ and 
    ${\mathcal H}(x)=\alpha x/\gamma +\left(\frac{\alpha}{m\delta}
-\frac{\alpha}{\gamma} \log (x)\right)$   therefore the 
    Lyapunov function reads
    \[
    L(x,y) = y - \frac{\gamma}{\delta} \log y - \frac{\alpha}{\delta} \left[ \left(1-\frac{1}{m}\cdot \frac{\gamma}{\delta} \right) \log x - x\right]-C,\]
where $m = \frac{1}{(\delta/\gamma - 1)}$ and
\[
C=\left[\frac{\gamma}{\delta} \left(1-\frac{\alpha}{\delta}\right) - \log \frac{\gamma}{\delta}\left(\frac{\gamma-\alpha}{\delta} - \frac{\alpha \gamma}{m \delta^2}\right) \right]
\]

We get the Jacobian matrix for the stationary point $P = ( 1- \frac{1}{m} \cdot \frac{\gamma}{\delta}, \gamma / \delta)$

\begin{align}
J(P) =
\begin{bmatrix}
0 & \delta \left(1 - \frac{1}{m} \cdot \frac{\gamma}{\delta} \right) \\
-\alpha \cdot \frac{\gamma}{\delta} & -\alpha \cdot \frac{1}{m} \cdot \frac{\gamma}{\delta}
\end{bmatrix}.
\end{align}

As the matrix is quite complex we will use determinant and trace of the Jacobian to determine the stability of this stability point. Let $J$ be Jacobian matrix for stationary point be then the characteristic equation of the matrix is

\begin{align}
\lambda^2 - (\text{tr}(J))\lambda + \det(J) = 0.
\end{align}

As such, we can determine the stability of the stationary point using trace and determinant. When
  $\text{tr}(J) < 0, \det(J) > 0, \Delta = \text{tr}(J)^2 - 4\det (J)< 0$ then the stationary point is an stable spiral. When 
  $\text{tr}(J) < 0, \det(J) > 0, \Delta = \text{tr}(J)^2 - 4\det (J)> 0$ then the stationary point is a stable node.

The trace and determinant of the Jacobian are:
\begin{align}
\text{tr}(J) = -\alpha \cdot \frac{1}{m} \cdot \frac{\gamma}{\delta} < 0,
\end{align}
\begin{align}
\det(J) = \alpha \gamma \left(1 - \frac{\gamma}{m \delta} \right).
\end{align}

Note that:
\begin{align}
1 - \frac{\gamma}{m \delta} = x^* > 0 \Rightarrow \det(J) > 0.
\end{align}

Since $\det(J) > 0$ and $\text{tr}(J) < 0$ it follows that the equilibrium point $P$ is locally asymptotically stable. It can be:
    \begin{itemize}
        \item a stable spiral (if $\Delta < 0$),
        \item a stable node (if $\Delta > 0$). 
    \end{itemize}
Because the stability of this equilibrium depends entirely on whether $\Delta$ is positive or negative, the system can smoothly transition from spiraling towards the equilibrium (oscillatory damping) to dropping directly into it (overdamped).

\begin{figure}[h]
\includegraphics[width=\textwidth,height=0.40\textheight,keepaspectratio]{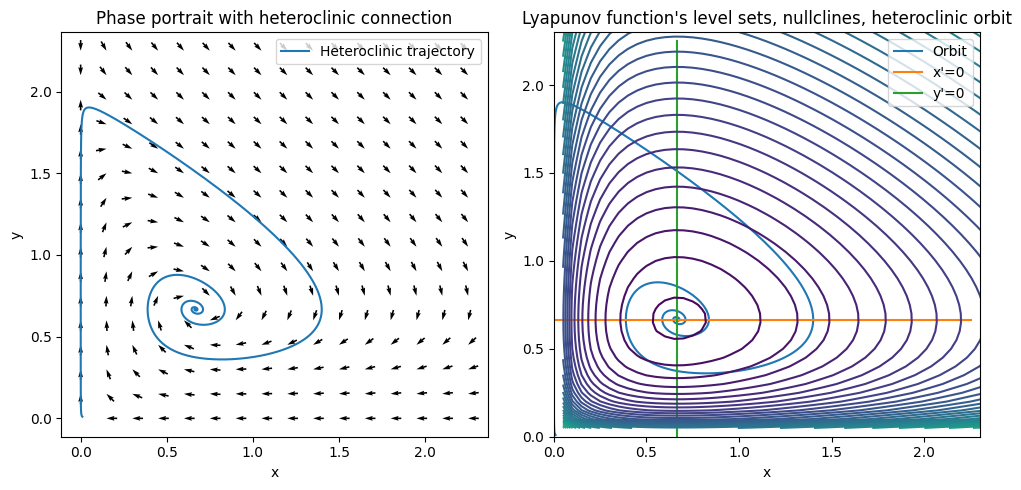}
\caption{Figures for biological predator-prey model with spiral behaviour}
\end{figure}

\section{Acknowledgements}
We would like to thank the students Magdalena Śmiałowicz and Łukasz R\c{e}bisz for the help with numerical simulations carried over with the aid of Python.

\bigskip


\noindent Robert Stańczy (corresponding author)\\
stanczr@math.uni.wroc.pl\\
ORCID: \url{https://orcid.org/0000-0003-3317-2012} \bigskip 

\noindent {\small
\noindent Uniwersytet Wrocławski\\
Instytut Matematyczny\\
Pl. Grunwaldzki 2/4, 50-384 Wrocław, Poland
}\bigskip

\noindent Dorota Bors\\
dorota.bors@wmii.uni.lodz.pl\\
ORCID: \url{https://orcid.org/0000-0001-9546-6978} \bigskip

\noindent {\small
\noindent University of Lodz\\
Faculty of Mathematics and Computer Science\\
Banacha 22, 90-238 Lodz, Poland
}\bigskip

\end{document}